\documentclass[reqno]{amsart}
\usepackage{geometry,amsmath,amssymb,bbm,color}
\newcommand{\assign}{:=}
\newcommand{\dueto}[1]{\textup{\textbf{(#1) }}}
\newcommand{\longhookrightarrow}{{\lhook\joinrel\relbar\joinrel\rightarrow}}
\newcommand{\mathd}{\mathrm{d}}
\newcommand{\tmop}[1]{\ensuremath{\operatorname{#1}}}
\newcommand{\tmtextit}[1]{{\itshape{#1}}}
\definecolor{grey}{rgb}{0.75,0.75,0.75}
\definecolor{orange}{rgb}{1.0,0.5,0.5}
\definecolor{brown}{rgb}{0.5,0.25,0.0}
\definecolor{pink}{rgb}{1.0,0.5,0.5}
\newtheorem{lemma}{Lemma}
\newtheorem{remark}{Remark}
\newtheorem{theorem}{Theorem}

\begin{document}

\title[Global Regularity of 2D Generalized MHD]{On Global Regularity of 
2D Generalized Magnetohydrodynamic Equations}
\author{Chuong V. Tran, Xinwei Yu, Zhichun Zhai}
\address{Chuong V. Tran: School of Mathematics and Statistics, 
University of St. Andrews, St Andrews KY16 9SS, United Kingdom}
\email{chuong@mcs.st-and.ac.uk}
\address{Xinwei Yu and Zhichun Zhai: Department of Mathematical and 
Statistical Sciences, University of Alberta, Edmonton, AB, T6G 2G1, Canada}
\email{xinweiyu@math.ualberta.ca, zhichun1@ualberta.ca}

\date{Apr. 16, 2012}


\begin{abstract}
  In this article we study the global regularity of 2D generalized
  magnetohydrodynamic equations (2D GMHD), in which the dissipation terms 
  are $- \nu \left( - \triangle \right)^{\alpha} u$ and $- \kappa \left( -
  \triangle \right)^{\beta} b$. We show that smooth solutions are global in
  the following three cases: $\alpha \geqslant 1 / 2, \beta \geqslant 1$; 
  $0 \leqslant \alpha < 1 / 2, 2 \alpha + \beta > 2$; $\alpha \geqslant 2, 
  \beta = 0$. We also show that in the inviscid case $\nu = 0$, if 
  $\beta > 1$, then smooth solutions are global as long as the direction 
  of the magnetic field remains smooth enough. \ 
\end{abstract}
\maketitle

\section{Introduction}


Recent mathematical studies of fluid mechanics have found it beneficial
to replace the Laplace operator $\triangle$, representing molecular 
diffusion, by fractional powers of $-\triangle$. For the magnetohydrodynamic 
(MHD) equations, this practice results in the generalized MHD (GMHD) system
\begin{eqnarray}
  u_t + u \cdot \nabla u & = & - \nabla p + b \cdot \nabla b - \nu \Lambda^{2
  \alpha} u,  \label{eq:GMHD-u}\\
  b_t + u \cdot \nabla b & = & b \cdot \nabla u - \kappa \Lambda^{2 \beta} b, 
  \label{eq:GMHD-b}\\
  \nabla \cdot u = \nabla \cdot b & = & 0,  \label{eq:GMHD-div}
\end{eqnarray}
which is the subject of the present study. Here 
$\nu, \kappa, \alpha, \beta \ge 0$ and $\Lambda = (-\triangle)^{1/2}$ 
is defined in terms of Fourier transform by 
\begin{equation}
\widehat{\Lambda f}(\xi) = \left| \xi \right| \widehat{f}(\xi).
\end{equation}
Equations (\ref{eq:GMHD-u}--\ref{eq:GMHD-div}) have been studied in
some detail by Wu \cite{Wu2003,Wu2004a} and Cao and Wu \cite{Cao2011},
with an emphasis on the issue of solution regularity.

The generalization of diffusion in the above manner has been implemented 
to other fluid systems, including the Navier--Stokes, Boussinesq, and
surface quasi-geostrophic equations (see e.g. \cite{Chae2011}, 
\cite{Chae2012pre}, \cite{Dong2010}, \cite{Hmidi2011}, \cite{Hmidi2010a}, 
\cite{Hmidi2010}, \cite{Li2010}). Studying these generalized equations
has enabled researchers to gain a deeper understanding of the strength 
and weaknesses of available mathematical methods and techniques, and, 
in some cases, motivated and inspired the invention of new methods. An 
illustrating example of the latter effect is the recent breakthroughs 
in the study of the surface quasi-geostrophic equations 
(\cite{Caffarelli2010}, \cite{Constantin2011}, \cite{Kiselev2010}, 
\cite{Kiselev2007}).



The problem of global well-posedness of the usual $n$-dimensional ($n$D) 
MHD (or GMHD with $\alpha,\beta \le 2$) equations, where $n \ge 3$, is 
highly challenging for obvious reasons. One is that the MHD equations 
include the Navier-Stokes (or Euler when $\nu = 0$) system as a special 
case (obtained by setting the initial magnetic field to zero), for which 
the issue of regularity has not been resolved. Another is that the quadratic 
coupling between $u$ and $b$ can introduce additional technical difficulties, 
even though this coupling may actually have some regularizing effects 
(see below). For $n = 2$, this coupling invalidates the vorticity
conservation, thereby becoming the main reason for the unavailability 
of a proof of global regularity for the ideal dynamics. Similar (but 
probably more manageable) situations arise when the 2D Euler equations 
are linearly coupled with the bouyancy equation in the Boussinesq system
or have a linear forcing term (\cite{Constantin2011}). \


So far the best result for the global regularity of the $n$D GMHD 
equations (\ref{eq:GMHD-u}--\ref{eq:GMHD-div}) has been derived in 
{\cite{Wu2011}}, where it has been proved that the system is globally 
regular as long as the following conditions
\begin{equation}
  \alpha \geqslant \frac{1}{2} + \frac{n}{4}, \hspace{2em} \beta > 0,
  \hspace{2em} \alpha + \beta \geqslant 1 + \frac{n}{2} ,
  \label{eq:GMHD-cond-Wu}
\end{equation}
are satisfied. Note that for simplicity of presentation, the above 
conditions have been given in slightly stronger forms than the exact 
result in {\cite{Wu2011}}, where the dissipation terms are allowed 
to be logarithmically weaker than $- \Lambda^{2 \alpha} u$ and 
$- \Lambda^{2 \beta} b$. Note also that for the case $n = 3$, conditions 
similar to (\ref{eq:GMHD-cond-Wu}) have been obtained in {\cite{Zhou2007}}, 
with $\beta > 0$ replaced by $\beta \geqslant 1$.


When $n \geqslant 3$, the result (\ref{eq:GMHD-cond-Wu}) is unlikely to 
be improved using current mathematical techniques. The reason is that 
the global regularity for the $n$D generalized Navier-Stokes equations
\begin{equation}
  u_t + u \cdot \nabla u = - \nabla p - \Lambda^{2 \alpha} u, \hspace{2em}
  \nabla \cdot u = 0. \label{eq:GNSE}
\end{equation}
is still unavailable for $\alpha < \frac{1}{2} + \frac{n}{4}$ (See 
\cite{Tao2009} for a proof of global regularity in the case of 
logarithmically weaker dissipation than $-\Lambda^{1+n/2}u$). On the
other hand, when $n = 2$, the availability of global regularity for the generalized Navier-Stokes equations (\ref{eq:GNSE}) for all $\alpha\geqslant 0$ suggests that the conditions in (\ref{eq:GMHD-cond-Wu}) could be excessive and may be weakened to some extent. In particular, it can be easily seen that the smoothness of either $u$ or $b$ guarantees 
that of the other and therefore of the system as a whole (\cite{TY2012}). 
Hence, global regularity could intuitively be possible with either 
$\nu=0$ or $\kappa=0$ for suitable conditions on $\beta$ or $\alpha$. 


In this article, we quantitatively confirm the above observations. 
More precisely, we show that when $n = 2$, the condition 
$\alpha \geqslant 1 = \frac{1}{2} + \frac{n}{4}$ is not needed
for the global regularity of the system. In particular, we focus on the
regime $\alpha < 1$ and show that the GMHD system is globally regular 
when $0 \leqslant \alpha < 1/2, 2\alpha + \beta > 2$ or when 
$\alpha \geqslant 1 / 2, \beta \geqslant 1$. We also prove global 
regularity for the case $\alpha\geqslant 2, \kappa=0$, thereby removing 
the technical condition $\beta > 0$. Furthermore, we study the inviscid 
case \ $\nu = 0$, $\kappa > 0$, and show that when $\beta > 1$, the 
GMHD system is globally regular as long as the magnetic lines are smooth 
enough. This result is consistent with numerical and experimental 
observations of the MHD dynamics, where the magnetic field appears to 
have the effect of ``suppressing'' the appearance of small scales in 
the fluid (see e.g. {\cite{Kraichnan1965}}), and as a consequence 
preventing the formation of singularities. Our finding is also 
consistent with a number of mathematical results exhibiting the 
regularizing effect on the streamlines and vortex lines in
Navier-Stokes and Euler dynamics (See e.g. {\cite{Constantin1993}},
{\cite{Deng2005}}, {\cite{Deng2006}}, {\cite{Vasseur2009}}).

The rest of this article is organized as follows. In Section
\ref{sec:main_results} we summarize the main results and give a brief
overview of the key ideas of their proofs. As these proofs use 
different methods for each case, we present them in separate
sections. Section \ref{sec:Case I} features the proof for global
regularity when $\alpha \geqslant 1 / 2, \beta \geqslant 1$. 
Sections \ref{Sec:Case II} and \ref{Sec:alpha=2} contain the proofs 
for the cases $0 \leqslant \alpha < 1 / 2, 2 \alpha + \beta > 2$ 
and $\alpha \geqslant 2, \beta = 0$, respectively. In Section 
\ref{sec:b_good} we prove global regularity under the assumption 
on the smoothness of magnetic lines.

Throughout this paper, we will set $\kappa = \nu = 1$ to simplify the
presentation. It is a standard exercise to adjust various constants 
to accommodate other values of $\kappa, \nu$, as long as both are
positive. We also identify the cases $\alpha = 0$ and $\beta = 0$ 
with $\nu = 0$ and $\kappa = 0$, respectively.

\section{Main Results}\label{sec:main_results}

Our first main result is the following global regularity theorem. \

\begin{theorem}
  \label{thm:global}Consider the GMHD equations
  (\ref{eq:GMHD-u}--\ref{eq:GMHD-div}) in 2D. Assume $\left( u_0, b_0 \right)
  \in H^k$ with $k > 2$. Then the system is globally regular for the following
  $\alpha, \beta$:
  \begin{itemize}
    \item $\alpha \geqslant 1 / 2, \beta \geqslant 1$;
    
    \item $0 \leqslant \alpha < 1 / 2, 2 \alpha + \beta > 2$;
    
    \item $\alpha \geqslant 2, \beta = 0$.
  \end{itemize}
\end{theorem}

\begin{remark}
  Combining the above theorem with the main result in {\cite{Wu2011}}, 
  we see that the 2D GMHD system is globally regular for all 
  $\alpha + \beta \geqslant 2$ except for $\alpha = 0, \beta = 2$. Thus 
  we have removed almost all technical conditions on $\alpha$ and $\beta$. 
\end{remark}

The three cases will be proved using different methods, as different types of
cancellation of the 2D GMHD system will be exploited. More specifically,
\begin{itemize}
  \item for $\alpha \geqslant 1 / 2, \beta \geqslant 1$, we apply standard
  $L^2$-based energy method, taking advantage of the special cancellation 
  that occurs for estimates in $H^1$.
  
  \item for $0 \leqslant \alpha < 1 / 2, 2 \alpha + \beta > 2$, 
  we derive a new non blow-up criterion in $L^p$ norm of the vorticity 
  $\omega = \nabla^\bot\cdot u = -\partial_2u_1 + \partial_1u_2$ 
  and then show that this criterion is indeed satisfied.
  
  \item for $\alpha \geqslant 2$, $\beta = 0$, we adapt the idea proposed
  in {\cite{Lei2009}}, carrying out a kind of ``weakly nonlinear'' energy
  estimate which takes advantage of the fact that in this case we have
  ``almost'' $H^1$ a priori bound. 
\end{itemize}
Our second main result is the following theorem dealing with the case 
$\nu = 0$ (for our purpose this is the same as $\alpha = 0$ since we do 
not impose any restriction on the size of the initial data).

\begin{theorem}
  \label{thm:b_good}Consider the  GMHD system
  (\ref{eq:GMHD-u}--\ref{eq:GMHD-div}) in 2D with $\alpha = 0$ and $\beta > 1$.
  Assume $\left( u_0, b_0 \right) \in H^k$ with $k > 2$. Then the system is
  globally regular if $\widehat{b} \assign \frac{b}{\left| b \right|} \in
  L^{\infty} \left( 0, T ; W^{2, \infty} \right)$.
\end{theorem}

\begin{remark}
  The condition on $\widehat{b}$ seems to be independent of the value of 
  $\beta$, in the sense that there is no $\beta_0$ such that as soon as 
  $\beta > \beta_0$, $\widehat{b}$ automatically belongs to 
  $L^{\infty} \left( 0, T ; W^{2,\infty} \right)$. 
\end{remark}

\textbf{Notation. }In the following we will use the standard function spaces $L^p$, $W^{k,p}$, $H^k$ whose norms are defined as
\[
\|f\|_{L^p} := \left(\int_{\mathbb{R}^2} \left|f \right|^p \mathd x\right)^{1/p},\qquad
\|f\|_{W^{k,p}} := \left( \sum_{|\alpha|=k} \|\partial^\alpha f\|_{L^p}^p \right)^{1/p}, \qquad
\|f\|_{H^k}:=\|f\|_{W^{k,2}}
\]
with standard modifications for the case $p=\infty$.

\section{Proof of Theorem \ref{thm:global} Case I: $\alpha \geqslant 1 / 2, \beta
\geqslant 1$.}\label{sec:Case I}

In this section we prove the first case of Theorem \ref{thm:global}. We apply
standard $L^2$-based energy estimates. The key idea here is to carry out 
the $H^1, H^2, H^k$ estimates successively to explore possible cancellations 
at each stage. We would like to mention that the cancellation at the $H^1$ 
stage has been observed before by several authors in the case $\beta = 1$ 
({\cite{Cao2011}},{\cite{Lei2009}}). The general case $\beta \geqslant 1$ 
is almost identical. However for completeness we still include detailed 
arguments.

\subsection{$H^1$ estimates ($L^2$ estimates for $\omega,
j$)}\label{subsec:H1-est}

\begin{lemma}
  \label{lem:H1-est}{\dueto{$H^1$ estimate}}Consider the 2D GMHD equations
  (\ref{eq:GMHD-u}--\ref{eq:GMHD-div}), where $\alpha \geqslant 0$ and 
  $\beta \geqslant 1$. Let 
  $\omega = \nabla^{\bot} \cdot u = - \partial_2 u_1 + \partial_1 u_2$ 
  and $j = \nabla^{\bot} \cdot b$. Let $u_0, b_0 \in H^1$. For fixed 
  $T > 0$ and $0 < t < T$, we have
  \begin{equation}
    \left\| \omega \right\|_{L^2}^2 \left( t \right) + \left\| j
    \right\|_{L^2}^2 \left( t \right) + \int_0^t \left( \left\|
    \Lambda^{\alpha} \omega \right\|_{L^2}^2 + \left\| \Lambda^{\beta} j
    \right\|_{L^2}^2 \right) \mathd \tau \leqslant C \left( u_0, b_0, T
    \right) .
  \end{equation}
\end{lemma}

\begin{proof}
  We first apply $\nabla^{\bot} \cdot$ to the GMHD equations
  (\ref{eq:GMHD-u}--\ref{eq:GMHD-div}) to obtain the governing equations for
  the vorticity $\omega$ and the current $j$:
  \begin{eqnarray}
    \omega_t + u \cdot \nabla \omega & = & b \cdot \nabla j - \Lambda^{2
    \alpha} \omega,  \label{eq:omega-H1}\\
    j_t + u \cdot \nabla j & = & b \cdot \nabla \omega + T \left( \nabla u,
    \nabla b \right) - \Lambda^{2 \beta} j.  \label{eq:j-H1}
  \end{eqnarray}
  Here
  \begin{equation}
    T \left( \nabla u, \nabla b \right) = {\color{black} 2 \partial_1 b_1 
    \left( \partial_1 u_2 + \partial_2 u_1 \right) + 2 \partial_2 u_2  \left(
    \partial_1 b_2 + \partial_2 b_1 \right)} . \label{eq:T-H1}
  \end{equation}
  Note that $T$ is bilinear in $\nabla u, \nabla b$ and therefore for any
  $k \geqslant 0$ we have
  \begin{equation}
    \left| \partial^k T \left( \nabla u, \nabla b \right) \right| \leqslant C
    \sum_{m = 0}^k \left| \nabla^{m + 1} u \right|  \left| \nabla^{k - m + 1}
    b \right|
  \end{equation}
  for some constant $C$ depending only on $k$.
  
  Multiplying (\ref{eq:omega-H1}) and (\ref{eq:j-H1}) by $\omega$ and $j$, 
  respectively, integrating, and adding the resulting equations together we obtain
  \begin{equation}
    \frac{1}{2}  \frac{\mathd}{\mathd t} \int_{\mathbbm{R^2}} \left( \omega^2
    + j^2 \right) \mathd x = \int_{\mathbbm{R}^2} T \left( \nabla u, \nabla b
    \right) j \mathd x - \int_{\mathbbm{R}^2} \left( \Lambda^{\alpha} \omega
    \right)^2 \mathd x - \int_{\mathbbm{R}^2} \left( \Lambda^{\beta} j
    \right)^2 \mathd x,
  \end{equation}
  where we have used the following consequences of $\nabla \cdot u = \nabla
  \cdot b = 0$:
  \begin{eqnarray}
    \int_{\mathbbm{R}^2} \left( u \cdot \nabla \omega \right) \omega \mathd x
    & = & 0 ; \\
    \int_{\mathbbm{R}^2} \left( u \cdot \nabla j \right) j \mathd x & = & 0 ;
    \\
    \int_{\mathbbm{R}^2} \left( b \cdot \nabla j \right) \omega \mathd x +
    \int_{\mathbbm{R}^2} \left( b \cdot \nabla \omega \right) j \mathd x & = &
    0. 
  \end{eqnarray}
  Note that all the terms involving derivatives of $\omega$ and $j$ -- 
  the ``worst" terms from energy estimate point of view -- disappear.
  
  Now recall the standard energy conservation which can be obtained by
  multiplying (\ref{eq:GMHD-u}) and (\ref{eq:GMHD-b}) \ by $u$ and $b$
  respectively, integrating, and applying the incompressibility condition
  (\ref{eq:GMHD-div}):
  \begin{equation}
    \frac{1}{2}  \frac{\mathd}{\mathd t} \int_{\mathbbm{R}^2} \left( u^2 + b^2
    \right) \mathd x + \int_{\mathbbm{R}^2} \left[\left( \Lambda^{\alpha} u
    \right)^2 + \left( \Lambda^{\beta} b \right)^2\right] \mathd x = 0.
  \end{equation}
  This gives
  \begin{equation}
    u \in L^{\infty} \left( 0, T ; L^2 \right) \cap L^2 \left( 0, T ;
    H^{\alpha} \right), \hspace{2em} b \in L^{\infty} \left( 0, T ; L^2
    \right) \cap L^2 \left( 0, T ; H^{\beta} \right) .
    \label{eq:energy-est-H1}
  \end{equation}
  As $\beta \geqslant 1$ by Sobolev embedding we easily get
  \begin{equation}
    b \in L^2 \left( 0, T ; H^1 \right) \Longrightarrow j \in L^2 \left( 0, T
    ; L^2 \right) . \label{eq:energy-j-H1}
  \end{equation}
  On the other hand we have
  \begin{equation}
    \left\| \Lambda j \right\|_{L^2} \leqslant C \left\| b \right\|_{L^2}^a
    \left\| \Lambda^{\beta} j \right\|^{1 - a}_{L^2}
  \end{equation}
  for
  \begin{equation}
    a = \frac{\beta - 1}{\beta + 1} .
  \end{equation}
  Using Young's inequality we obtain
  \begin{equation}
    \left\| \Lambda j \right\|_{L^2}^2 \leqslant a \left\| b \right\|_{L^2}^2
    + \left( 1 - a \right)  \left\| \Lambda^{\beta} j \right\|_{L^2}^2
    \Longrightarrow \left\| \Lambda^{\beta} j \right\|_{L^2}^2 \geqslant
    \frac{1}{1 - a}  \left\| \Lambda j \right\|_{L^2}^2 - \frac{a}{1 - a} 
    \left\| b \right\|_{L^2} .
  \end{equation}
  It is worth emphasizing that the above calculation remains valid even 
  when $a = 0$, that is $\beta = 1$.
  
  This leads us to
  \begin{eqnarray}
    \frac{\mathd}{\mathd t} \left( \left\| \omega \right\|_{L^2}^2 + \left\|
    j \right\|_{L^2}^2 \right) & \leqslant & C \int_{\mathbbm{R}^2} \left|
    \nabla u \right|  \left| \nabla b \right|  \left| j \right| \mathd x -
    \frac{1}{\left( 1 - a \right)}  \left\| \Lambda j \right\|_{L^2}
    \nonumber\\
    &  & + \frac{a}{\left( 1 - a \right)}  \left\| b \right\|_{L^2} - 2
    \left\| \Lambda^{\alpha} \omega \right\|_{L^2}^2 - \left\| \Lambda^{\beta}
    j \right\|_{L^2}^2.
  \end{eqnarray}
  By H\"older's inequality, the trilinear term satisfies
  \begin{equation}
    \int_{\mathbbm{R}^2} \left| \nabla u \right|  \left| \nabla b \right| 
    \left| j \right| \mathd x \leqslant \left\| \nabla u \right\|_{L^2} 
    \left\| \nabla b \right\|_{L^4}  \left\| j \right\|_{L^4} .
  \end{equation}
  Owing to the relations
  \begin{equation}
    \nabla u = \nabla \left( - \triangle \right)^{- 1} \nabla^{\bot} \omega
    \text{ \tmop{and} } \nabla b = \nabla \left( - \triangle \right)^{- 1}
    \nabla^{\bot} j
  \end{equation}
  we have, following standard Fourier multiplier theory (see e.g.
  {\cite{Stein1970}}),
  \begin{equation}
    \left\| \nabla u \right\|_{L^2} \leqslant C \left\| \omega \right\|_{L^2}
    \text{ \tmop{and} } \left\| \nabla b \right\|_{L^4} \leqslant C \left\| j
    \right\|_{L^4}
  \end{equation}
  for some absolute constant $C$. It follows that
  \begin{equation}
    \int_{\mathbbm{R}^2} \left| \nabla u \right|  \left| \nabla b \right| 
    \left| j \right| \mathd x \leqslant C \left\| \omega \right\|_{L^2} 
    \left\| j \right\|_{L^4}^2 .
  \end{equation}
  Next, application of the Gagliardo-Nirenberg inequality
  \begin{equation}
    \left\| j \right\|_{L^4} \leqslant C \left\| j \right\|_{L^2}^{1 / 2} 
    \left\| \Lambda j \right\|_{L^2}^{1 / 2}
  \end{equation}
  yields
  \begin{equation}
    \int_{\mathbbm{R}^2} \left| \nabla u \right|  \left| \nabla b \right| 
    \left| j \right| \mathd x \leqslant C \left\| \omega \right\|_{L^2} 
    \left\| j \right\|_{L^2}  \left\| \Lambda j \right\|_{L^2} \leqslant C
    \left( \varepsilon \right)  \left\| j \right\|_{L^2}^2  \left\| \omega
    \right\|_{L^2}^2 + \varepsilon \left\| \Lambda j \right\|_{L^2}^2,
  \end{equation}
  where Young's inequality has been used. Here $\varepsilon$ is a small
  positive number that will be chosen later.


  Summarizing the above, we have
  \begin{eqnarray}
    \frac{\mathd}{\mathd t} \left( \left\| \omega \right\|_{L^2}^2 + \left\| j
    \right\|_{L^2}^2 \right) + \left\| \Lambda^{\alpha} \omega
    \right\|_{L^2}^2 + \left\| \Lambda^{\beta} j \right\|_{L^2} & \leqslant &
    C \left( \varepsilon \right)  \left\| j \right\|_{L^2}^2  \left\| \omega
    \right\|_{L^2}^2 + C \varepsilon \left\| \Lambda j \right\|^2_{L^2}
    \nonumber\\
    &  & - \frac{1}{\left( 1 - a \right)}  \left\| \Lambda j \right\|_{L^2}^2
    + \frac{a}{\left( 1 - a \right)}  \left\| b \right\|_{L^2} . 
  \end{eqnarray}
  Taking $\varepsilon$ small enough so that $C \varepsilon < \frac{1}{1 - a}$,
  we obtain
  \begin{equation}
    \frac{\mathd}{\mathd t} \left( \left\| \omega \right\|_{L^2}^2 + \left\| j
    \right\|_{L^2}^2 \right) + \left\| \Lambda^{\beta} j \right\|_{L^2}^2 +
    \left\| \Lambda^{\alpha} \omega \right\|_{L^2}^2 \leqslant C \left(
    \varepsilon \right)  \left\| j \right\|_{L^2}^2  \left\| \omega
    \right\|_{L^2}^2 + \frac{a}{1 - a}  \left\| b \right\|_{L^2} .
  \end{equation}
  As $\left\| b \right\|_{L^2}$ is uniformly bounded in $t$, and $\left\| j
  \right\|_{L^2}^2 \in L^1 \left( 0, T \right)$ ((\ref{eq:energy-est-H1} --
  \ref{eq:energy-j-H1})), the proof is completed.
\end{proof}


\begin{remark}
  \label{rem:reduction}Note that the above proof can be shortened by skipping
  the steps
  \begin{equation}
    \left\| \Lambda j \right\|_{L^2} \leqslant \left\| b \right\|_{L^2}^a
    \left\| \Lambda^{\beta} j \right\|^{1 - a}_{L^2}
  \end{equation}
  and
  \begin{equation}
    \left\| \Lambda^{\beta} j \right\|_{L^2}^2 \geqslant \frac{1}{1 - a} 
    \left\| \Lambda j \right\|_{L^2}^2 - \frac{a}{1 - a}  \left\| b
    \right\|_{L^2}
  \end{equation}
  and directly applying the Gagliardo-Nirenberg inequality
  \begin{equation}
    \left\| j \right\|_{L^4} \leqslant \left\| j \right\|_{L^2}^{a_1}  \left\|
    \Lambda^{\beta} b \right\|_{L^2}^{a_2}  \left\| \Lambda^{\beta} j
    \right\|_{L^2}^{a_3}
  \end{equation}
  for appropriate $a_1, a_2, a_3$, and then use Young's inequality.
  However we choose to first reduce the general situation $\beta \geqslant 1$
  to the particular one $\beta = 1$ to illustrate the following observation:
  For our problem, to prove regularity for $\alpha \geqslant \alpha_0, \beta
  \geqslant \beta_0$ using energy method, it suffices to do so for $\alpha =
  \alpha_0, \beta = \beta_0$. Such reduction significantly reduces the number
  of parameters in higher Sobolev norm estimates and makes the presentation
  much more transparent, as we will see in the following $H^2$ estimate.
\end{remark}

\subsection{$H^2$ estimates ($H^1$ estimates for $\omega,
j$)}\label{subsec:H2-est}

With $H^1$ estimates at hand, we can move on to $H^2$ estimates.
Differentiating (\ref{eq:omega-H1}--\ref{eq:j-H1}) we reach
\begin{equation}
  \left( \partial_i \omega \right)_t + u \cdot \nabla \left( \partial_i \omega
  \right) = - \left( \partial_i u \right) \cdot \nabla \omega + \left(
  \partial_i b \right) \cdot \nabla j + b \cdot \nabla \left( \partial_i j
  \right) - \Lambda^{2 \alpha} \left( \partial_i \omega \right)
\end{equation}
\begin{equation}
  \left( \partial_i j \right)_t + u \cdot \nabla \left( \partial_i j \right) =
  - \left( \partial_i u \right) \cdot \nabla j + \left( \partial_i b \right)
  \cdot \nabla \omega + b \cdot \nabla \left( \partial_i \omega \right) +
  \partial_i \left( T \left( \nabla u, \nabla b \right) \right) - \Lambda^{2
  \beta} \left( \partial_i j \right) .
\end{equation}
This gives the following integral relation:
\begin{eqnarray}
  \frac{\mathd}{\mathd t} \int_{\mathbbm{R}^2} \frac{\left( \partial_i \omega
  \right)^2 + \left( \partial_i j \right)^2}{2} \mathd x& = & - \int_{\mathbbm{R}^2}
  \left[ \left( \partial_i u \right) \cdot \nabla \omega \right] \left(
  \partial_i \omega \right) \mathd x + \int_{\mathbbm{R}^2} \left[ \left(
  \partial_i b \right) \cdot \nabla j \right]  \left( \partial_i \omega
  \right) \mathd x \nonumber\\
  &  & - \int_{\mathbbm{R}^2} \left[ \left( \partial_i u \right) \cdot \nabla
  j \right]  \left( \partial_i j \right) \mathd x + \int_{\mathbbm{R}^2}
  \left[ \left( \partial_i b \right) \cdot \nabla \omega \right]  \left(
  \partial_i j \right) \mathd x \nonumber\\
  &  & + \int_{\mathbbm{R}^2} \left[ \partial_i \left( T \left( \nabla u,
  \nabla b \right) \right) \right]  \left( \partial_i j \right) \mathd x\nonumber\\
  &&-\int_{\mathbbm{R}^2} \left( \Lambda^{\alpha} \partial_i \omega \right)^2
  \mathd x - \int_{\mathbbm{R}^2} \left( \Lambda^{\beta} \partial_i j
  \right)^2 \mathd x. 
\end{eqnarray}
after taking advantage of $\nabla \cdot u = \nabla \cdot b = 0$.

Summing up $i = 1, 2$, we reach
\begin{equation}
  \frac{\mathd}{\mathd t} \left( \left\| \nabla \omega \right\|_{L^2}^2 +
  \left\| \nabla j \right\|_{L^2}^2 \right) \leqslant C \left( I_1 + I_2 + I_3
  + I_4 + I_5 \right) - 2 \left\| \Lambda^{\alpha} \nabla \omega
  \right\|_{L^2}^2 - 2 \left\| \Lambda^{\beta} \nabla j \right\|_{L^2}^2
\end{equation}
with $C$ an absolute constant, and
\begin{eqnarray}
  I_1 & = & \int_{\mathbbm{R}^2} \left| \nabla u \right|  \left| \nabla \omega
  \right|^2 \mathd x ;  \label{eq:I1}\\
  I_2 & = & \int_{\mathbbm{R}^2} \left| \nabla b \right|  \left| \nabla j
  \right|  \left| \nabla \omega \right| \mathd x ;  \label{eq:I2}\\
  I_3 & = & \int_{\mathbbm{R}^2} \left| \nabla u \right|  \left| \nabla j
  \right|^2 \mathd x ;  \label{eq:I3}\\
  I_4 & = & \int_{\mathbbm{R}^2} \left| \nabla b \right|  \left| \nabla \omega
  \right|  \left| \nabla j \right| \mathd x ;  \label{eq:I4}\\
  I_5 & = & \int_{\mathbbm{R}^2} \left[ \left| \nabla^2 u \right|  \left|
  \nabla b \right| + \left| \nabla u \right|  \left| \nabla^2 b \right|
  \right]  \left| \nabla j \right| \mathd x.  \label{eq:I5}
\end{eqnarray}
We estimate these quantities one by one. As discussed in Remark 
\ref{rem:reduction}, we only need to carry out the estimates for 
the case $\alpha = 1 / 2, \beta = 1$.

There are four different cases ($I_2$ and $I_4$ are identical).
\begin{itemize}

  \item Estimating $I_1 = \int_{\mathbbm{R}^2} \left| \nabla u \right|  \left|
  \nabla \omega \right|^2 \mathd x$.
  
  First, by H\"older's inequality we have
  \begin{equation}
    I_1 \leqslant \left\| \nabla u \right\|_{L^3}  \left\| \nabla \omega
    \right\|_{L^3}^2 \leqslant C \left\| \omega \right\|_{L^3}  \left\| \nabla
    \omega \right\|_{L^3}^2 .
  \end{equation}
  Consider the following Gagliardo-Nirenberg inequalities.
  \begin{eqnarray}
    \left\| \nabla \omega \right\|_{L^3} & \leqslant & C \left\| \Lambda^{1 /
    2} \omega \right\|_{L^2}^{1 / 6}  \left\| \Lambda^{1 / 2} \nabla \omega
    \right\|_{L^2}^{5 / 6} ;  \label{eq:H2-est-01}\\
    \left\| \nabla \omega \right\|_{L^3} & \leqslant & C \left\| \nabla \omega
    \right\|^{1 / 3}_{L^2}  \left\| \Lambda^{1 / 2} \nabla \omega
    \right\|_{L^2}^{2 / 3} ;  \label{eq:H2-est-02}\\
    \left\| \omega \right\|_{L^3} & \leqslant & C \left\| \omega
    \right\|_{L^2}^{7 / 9}  \left\| \Lambda^{1 / 2} \nabla \omega
    \right\|_{L^2}^{2 / 9} .  \label{eq:H2-est-03}
  \end{eqnarray}
  Equations (\ref{eq:H2-est-01}) and (\ref{eq:H2-est-02}) imply
  \begin{equation}
    \left\| \nabla \omega \right\|_{L^3} = \left\| \nabla \omega
    \right\|_{L^3}^{2 / 3}  \left\| \nabla \omega \right\|_{L^3}^{1 / 3}
    \leqslant C \left\| \Lambda^{1 / 2} \omega \right\|_{L^2}^{1 / 9}  \left\|
    \nabla \omega \right\|_{L^2}^{1 / 9}  \left\| \Lambda^{1 / 2} \nabla
    \omega \right\|_{L^2}^{7 / 9} . \label{eq:H2-est-04}
  \end{equation}
  Now (\ref{eq:H2-est-03}) and (\ref{eq:H2-est-04}) gives
  \begin{equation}
    I_1 \leqslant C \left\| \omega \right\|_{L^3}  \left\| \nabla \omega
    \right\|_{L^3}^2 \leqslant C \left\| \omega \right\|_{L^2}^{7 / 9} 
    \left\| \Lambda^{1 / 2} \omega \right\|_{L^2}^{2 / 9}  \left\| \nabla
    \omega \right\|_{L^2}^{2 / 9}  \left\| \Lambda^{1 / 2} \nabla \omega
    \right\|_{L^2}^{16 / 9}.
  \end{equation}
  Applying Young's inequality we get
  \begin{equation}
    I_1 \leqslant C \left( \varepsilon \right)  \left\| \omega
    \right\|_{L^2}^7  \left\| \Lambda^{1 / 2} \omega \right\|_{L^2}^2  \left\|
    \nabla \omega \right\|_{L^2}^2 + \varepsilon \left\| \Lambda^{1 / 2}
    \nabla \omega \right\|_{L^2}^2 .
  \end{equation}
  Here $\varepsilon$ can be taken as small as we want and will be specified
  later.
  
  \item Estimating $I_2 = I_4 = \int_{\mathbbm{R}^2} \left| \nabla b \right| 
  \left| \nabla j \right|  \left| \nabla \omega \right| \mathd x$.
  
  Using H\"older's inequality we have
  \begin{equation}
    \int_{\mathbbm{R}^2} \left| \nabla b \right|  \left| \nabla j \right| 
    \left| \nabla \omega \right| \mathd x \leqslant \left\| \nabla b
    \right\|_{L^4}  \left\| \nabla j \right\|_{L^4}  \left\| \nabla \omega
    \right\|_{L^2} \leqslant C \left\| j \right\|_{L^4} \left\| \nabla j
    \right\|_{L^4}  \left\| \nabla \omega \right\|_{L^2} .
  \end{equation}
  Applying the Gagliardo-Nirenberg inequalities
  \begin{equation}
    \left\| j \right\|_{L^4} \leqslant C \left\| j \right\|_{L^2}^{1 / 2} 
    \left\| \nabla j \right\|_{L^2}^{1 / 2} ; \hspace{2em} \left\| \nabla j
    \right\|_{L^4} \leqslant C \left\| \nabla j \right\|_{L^2}^{1 / 2} 
    \left\| \Lambda \nabla j \right\|_{L^2}^{1 / 2} \label{GN}
  \end{equation}
  yields
  \begin{equation}
    \int_{\mathbbm{R}^2} \left| \nabla b \right|  \left| \nabla j \right| 
    \left| \nabla \omega \right| \mathd x \leqslant C \left\| j
    \right\|_{L^2}^{1 / 2}  \left\| \nabla j \right\|_{L^2}  \left\| \Lambda
    \nabla j \right\|_{L^2}^{1 / 2}  \left\| \nabla \omega \right\|_{L^2} .
  \end{equation}
  Applying Young's inequality further yields
  \begin{equation}
    \int_{\mathbbm{R}^2} \left| \nabla b \right|  \left| \nabla j \right| 
    \left| \nabla \omega \right| \mathd x \leqslant C \left( \varepsilon
    \right) \left\| j \right\|_{L^2}^2 + \left\| \nabla j \right\|_{L^2}^2 
    \left\| \nabla \omega \right\|_{L^2}^2 + \varepsilon \left\| \Lambda
    \nabla j \right\|_{L^2}^2 .
  \end{equation}

  \item Estimating $I_3 = \int_{\mathbbm{R}^2} \left| \nabla u \right|  \left|
  \nabla j \right|^2 \mathd x$.
  
  Using H\"older's inequality we have
  \begin{equation}
    \int_{\mathbbm{R}^2} \left| \nabla u \right|  \left| \nabla j \right|^2
    \mathd x \leqslant \left\| \nabla u \right\|_{L^2}  \left\| \nabla j
    \right\|_{L^4}^2 \leqslant C \left\| \omega \right\|_{L^2}  \left\| \nabla
    j \right\|_{L^4}^2 .
  \end{equation}
  Now using the second Gagliardo-Nirenberg inequality in (\ref{GN})
  and Young's inequality we get
  \begin{equation}
    I_3 \leqslant C \left\| \omega \right\|_{L^2}  \left\| \nabla j
    \right\|_{L^2}  \left\| \Lambda \nabla j \right\|_{L^2} \leqslant C \left(
    \varepsilon \right)  \left\| \omega \right\|_{L^2}^2  \left\| \nabla j
    \right\|_{L^2}^2 + \varepsilon \left\| \Lambda \nabla j \right\|_{L^2}^2 .
  \end{equation}
  \item Estimating $I_5 = \int_{\mathbbm{R}^2} \left[ \left| \nabla^2 u \right| 
  \left| \nabla b \right| + \left| \nabla u \right|  \left| \nabla^2 b \right|
  \right]  \left| \nabla j \right| \mathd x$.
  
  We write
  \begin{equation}
    I_5 = I_{51} + I_{52} \assign \int_{\mathbbm{R}^2} \left| \nabla^2 u
    \right|  \left| \nabla b \right|  \left| \nabla j \right| \mathd x +
    \int_{\mathbbm{R}^2} \left| \nabla u \right|  \left| \nabla^2 b \right| 
    \left| \nabla j \right| \mathd x.
  \end{equation}
  It is clear that $I_{51}$ can be estimated similar to $I_2$ while $I_{52}$
  can be estimated similar to $I_3$. 
\end{itemize}
\begin{remark}
  \label{rem:I2-I5}We would like to emphasize that the assumption $\alpha
  \geqslant 1 / 2$ is only needed for the estimation of $I_1$. The estimates
  $I_2 - I_5$ only require $\alpha \geqslant 0, \beta \geqslant 1$. 
\end{remark}

Putting the above results together, we have
\begin{eqnarray}
  \frac{\mathd}{\mathd t} \left( \left\| \nabla \omega \right\|_{L^2}^2 +
  \left\| \nabla j \right\|_{L^2}^2 \right) & \leqslant & C \left( \varepsilon
  \right)  \left[ \left\| \omega \right\|_{L^2}^7  \left\| \Lambda^{1 / 2}
  \omega \right\|_{L^2}^2 + \left\| \omega \right\|_{L^2}^2 + \left\| \nabla j
  \right\|_{L^2}^2 + 1 \right]  \left( \left\| \nabla \omega \right\|_{L^2}^2
  + \left\| \nabla j \right\|_{L^2}^2 \right) \nonumber\\
  &  & + C \left( \varepsilon \right)  \left\| j \right\|_{L^2}^2 - 2 \left\|
  \Lambda^{1 / 2} \nabla \omega \right\|_{L^2}^2 -2 \left\| \Lambda \nabla j
  \right\|_{L^2}^2 \nonumber\\
  &  & + C \varepsilon \left( \left\| \Lambda^{1 / 2} \nabla \omega
  \right\|_{L^2}^2 + \left\| \Lambda \nabla j \right\|_{L^2}^2 \right) . 
\end{eqnarray}
Taking $\varepsilon$ small enough so that $C \varepsilon < 1$ we have
\begin{eqnarray}
  \frac{\mathd}{\mathd t} \left( \left\| \nabla \omega \right\|_{L^2}^2 +
  \left\| \nabla j \right\|_{L^2}^2 \right) & \leqslant & C \left( \varepsilon
  \right)  \left[ \left\| \omega \right\|_{L^2}^7  \left\| \Lambda^{1 / 2}
  \omega \right\|_{L^2}^2 + \left\| \omega \right\|_{L^2}^2 + \left\| \nabla j
  \right\|_{L^2}^2 + 1 \right]  \left( \left\| \nabla \omega \right\|_{L^2}^2
  + \left\| \nabla j \right\|_{L^2}^2 \right) \nonumber\\
  &  & + C \left( \varepsilon \right)  \left\| j \right\|_{L^2}^2 - 
  \left( \left\| \Lambda^{1 / 2} \nabla \omega \right\|_{L^2}^2 + \left\|
  \Lambda \nabla j \right\|_{L^2}^2 \right). \label{eq:est-H2}
\end{eqnarray}
Recall that
\begin{equation}
  \left\| \Lambda^{1 / 2} \omega \right\|_{L^2}, \left\| \nabla j
  \right\|_{L^2} \in L^2 \left( 0, T \right) ; \hspace{2em} \left\| \omega
  \right\|_{L^2}, \left\| j \right\|_{L^2} \in L^{\infty} \left( 0, T \right)
\end{equation}
thanks to the $H^1$ estimate. This, together with (\ref{eq:est-H2}), 
implies
\begin{equation}
  \nabla \omega, \nabla j \in L^{\infty} \left( 0, T ; L^2 \right) .
\end{equation}
Combining with the $H^1$ estimate, we have the following $H^2$ estimate:
\begin{equation}
  \left\| \omega \right\|_{H^1} + \left\| j \right\|_{H^1} \in L^{\infty}
  \left( 0, T \right) .
\end{equation}
\subsection{$H^k$ estimates}\label{subsec:Hk-est}

An argument which by now is standard (see for example {\cite{Lei2009}}) 
generalizes the
classical BKM-type blow-up criterion ({\cite{Caflisch1997}}) to
\begin{equation}
  \text{The MHD system stays regular beyond }T\text{ if and only if } \int_0^T \left( \left\| \omega
  \right\|_{\tmop{BMO}} + \left\| j \right\|_{\tmop{BMO}} \right)\mathd t <\infty. 
\end{equation}
Using the embedding
\begin{equation}
  H^1 \longhookrightarrow \tmop{BMO}
\end{equation}
in 2D, we see that
\begin{equation}
  \left\| \omega \right\|_{H^1} + \left\| j \right\|_{H^1} \in L^{\infty}
  \left( 0, T \right) \Longrightarrow \left\| \omega \right\|_{\tmop{BMO}} +
  \left\| j \right\|_{\tmop{BMO}} \in L^{\infty} \left( 0, T \right)
\end{equation}
and consequently all $H^k$ norms are bounded. This completes the proof of the
first case.

\section{Proof of Theorem \ref{thm:global} Case II: $0 \leqslant \alpha < 1 / 2,
2 \alpha + \beta > 2$. }\label{Sec:Case II}

To prove global regularity in this case, we first derive a blow-up criterion
in $\left\| \omega \right\|_{L^p}$ for appropriate $p$, then obtain a priori
estimate for $\left\| \omega \right\|_{L^p}$. Note that in this case we have
$\beta > 1$ and Lemma \ref{lem:H1-est} together with the embedding
\begin{equation}
  H^{\beta} \longhookrightarrow L^{\infty}
\end{equation}
in 2D already gives $j \in L^2 \left( 0, T ; L^{\infty} \right) \longhookrightarrow
L^1 \left( 0, T ; \tmop{BMO} \right)$.

\begin{lemma}
  \label{lem:blowup-omega-p}Assume $0< \alpha <1/2, \beta > 1$. The GMHD
  system (\ref{eq:GMHD-u}--\ref{eq:GMHD-div}) is regular if $\omega \in L^p$ 
  for any $p > \frac{1}{\alpha}$. 
\end{lemma}

\begin{proof}
  As we have $\beta > 1$, we already have the following $H^1$ estimates thanks
  to Lemma \ref{lem:H1-est}:
  \begin{equation}
    \omega \in L^{\infty} \left( 0, T ; L^2 \right) \cap L^2 \left( 0, T ;
    H^{\alpha} \right) ; \hspace{2em} j \in L^{\infty} \left( 0, T ; L^2
    \right) \cap L^2 \left( 0, T ; H^{\beta} \right) .
  \end{equation}
  Now arguing similarly as in Sections \ref{subsec:H2-est} and \ref{subsec:Hk-est}, we
  see that all we need to do is to bound $I_1 - I_5$ as defined in
  (\ref{eq:I1}--\ref{eq:I5}). Furthermore, we note that the estimates for $I_2
  - I_5$ can be done similarly to that in Section \ref{subsec:H2-est}, as explained in
  Remark \ref{rem:I2-I5}. The only estimate that needs to be done differently
  is that of $I_1 = \int_{\mathbbm{R}^2} \left| \omega \right|  \left| \nabla
  \omega \right|^2 \mathd x$.
  
  For that purpose, we first apply H\"older's inequality to $I_1$ to obtain
  \begin{equation}
    \int_{\mathbbm{R}^2} \left| \omega \right|  \left| \nabla \omega \right|^2
    \mathd x \leqslant \left\| \omega \right\|_{L^{p_1}}  \left\| \nabla
    \omega \right\|_{L^{2 q_1}}^2
  \end{equation}
  for $p_1, q_1$ satisfy
  \begin{equation}
    p_1 > \frac{1}{\alpha}, \hspace{2em} \frac{1}{p_1} + \frac{1}{q_1} = 1.
  \end{equation}
  Next we use the following Gagliardo-Nirenberg type inequalities:
  \begin{eqnarray}
    \left\| \nabla \omega \right\|_{L^{2 q_1}} & \leqslant & C \left\|
    \Lambda^{\alpha} \omega \right\|_{L^2}^{\xi}  \left\| \Lambda^{\alpha}
    \nabla \omega \right\|_{L^2}^{1 - \xi} \text{\tmop{with} } \xi = \alpha -
    \frac{1}{p_1} = \alpha \left( 1 - \frac{1}{p_1 \alpha} \right) ; \\
    \left\| \nabla \omega \right\|_{L^{2 q_1}} & \leqslant & C \left\| \nabla
    \omega \right\|_{L^2}^{\eta} \left\| \Lambda^{\alpha} \nabla \omega
    \right\|_{L^2}^{1 - \eta} \text{\tmop{with} }  \eta = 1 - \frac{1}{p_1
    \alpha} . 
  \end{eqnarray}
  Note that as long as $p_1 > \frac{1}{\alpha}$ both $\xi, \eta \in \left( 0,
  1 \right)$. Now setting
  \begin{equation}
    a = \frac{\alpha}{1 + \alpha}  \left( 1 - \frac{1}{p_1 \alpha} \right),  
  \label{eq:def-a}
  \end{equation}
which satisfies $0<a<1/3$ owing to $0<\alpha<1/2$ and $p_1>1/\alpha$, we have
  \begin{equation}
    \left\| \nabla \omega \right\|_{L^{2 q_1}} = \left\| \nabla \omega
    \right\|_{L^{2 q_1}}^{1 / \left( 1 + \alpha \right)}  \left\| \nabla
    \omega \right\|_{L^{2 q_1}}^{\alpha / \left( 1 + \alpha \right)} \leqslant
    C \left\| \Lambda^{\alpha} \omega \right\|_{L^2}^a  \left\| \nabla \omega
    \right\|_{L^2}^a \left\| \Lambda^{\alpha} \nabla \omega \right\|_{L^2}^{1
    - 2 a} .
  \end{equation}
  Next we apply the following Gagliardo-Nirenberg inequality
  \begin{equation}
    \left\| \omega \right\|_{L^{p_1}} \leqslant C \left\| \omega
    \right\|_{L^p}^{1 - 2 a}  \left\| \Lambda^{\alpha} \nabla \omega
    \right\|_{L^2}^{2 a},
  \end{equation}
  where $a$ is given by (\ref{eq:def-a}) and $p < p_1$. The exact value 
  of $p$ can be written down but what is important here is that 
  $p > \frac{1}{\alpha}$, as can be seen from the following manipulation 
  of the scaling relation:
  \begin{equation}
    - \frac{2}{p_1} = \left( 1 - 2 a \right)  \left( - \frac{2}{p} \right) + 2
    a \alpha \Longrightarrow - \frac{1}{p_1} = \left( 1 - 2 a \right)  \left(
    - \frac{1}{p} \right) + a \alpha . \label{eq:scaling-relation}
  \end{equation}
  Writing (\ref{eq:def-a}) as $a = \frac{1}{1 + \alpha}  \left( \alpha -
  \frac{1}{p_1} \right)$ and then adding $\alpha$ to both sides of
  (\ref{eq:scaling-relation}), we reach
  \begin{equation}
    \alpha - \frac{1}{p} = \frac{1 - 3 \alpha / \left( \alpha + 1 \right)}{1 -
    2 a}  \left( \alpha - \frac{1}{p_1} \right) .
  \end{equation}
  Recalling $\alpha < 1 / 2$, we see that $\alpha - 1 / p > 0$ if and 
  only if $\alpha - 1 / p_1 > 0$.
  
  Combining the above, and applying Young's inequality, we see that $I_1$ 
  can be bounded as
  \begin{eqnarray}
    I_1 \leqslant \left\| \omega \right\|_{L^{p_1}}  \left\| \nabla \omega
    \right\|_{L^{2 q_1}}^2 & \leqslant & C \left\| \omega \right\|_{L^p}^{1 - 2
    a}  \left( \left\| \Lambda^{\alpha} \omega \right\|_{L^2}^a  \left\|
    \nabla \omega \right\|_{L^2}^a \left\| \Lambda^{\alpha} \nabla \omega
    \right\|_{L^2}^{1 - a} \right)^2 \nonumber\\
    & \leqslant & C \left( \varepsilon \right)  \left\| \omega
    \right\|_{L^p}^{\left( 1 - 2 a \right) / a}  \left\| \Lambda^{\alpha}
    \omega \right\|_{L^2}^2  \left\| \nabla \omega \right\|_{L^2}^2 +
    \varepsilon \left\| \Lambda^{\alpha} \nabla \omega \right\|_{L^2}^2 . 
  \end{eqnarray}
  Now it is clear that once $\left\| \omega \right\|_{L^p} \in L^{\infty}
  \left( 0, T \right)$, we can obtain $H^2$ estimate as in Section
  \ref{subsec:H2-est}, and global regularity follows as in
Section   \ref{subsec:Hk-est}.
  
  Finally, if $\left\| \omega \right\|_{L^q}$ is bounded for some $q >
  \frac{1}{\alpha} > 2$, then together with the $H^1$ estimate $\omega \in
  L^{\infty} \left( 0, T ; L^2 \right)$ we see that
  \begin{equation}
    \left\| \omega \right\|_{L^r} \in L^{\infty} \left( 0, T \right)
    \hspace{2em} \forall r \in \left[ 2, q \right] .
  \end{equation}
  Now we can simply take $p_1 = q$ in the above inequalities, then since $p < p_1$ we have the uniform
  boundedness of $\left\| \omega \right\|_{L^p}$ and global regularity
  follows. 
\end{proof}

\begin{remark}
  The case $\alpha = 0$ (which we identify with the case $\nu = 0$) is
  trivial. By our assumption $2 \alpha + \beta > 2$ we have $\beta > 2$, 
  which gives $\nabla j \in L^2 \left( 0, T ; L^{\infty} \right)$. This
  result, together with the vorticity equation
  \begin{equation}
    \omega_t + u \cdot \nabla \omega = b \cdot \nabla j,
  \end{equation}
  implies $\omega \in L^{\infty} \left( 0, T ; L^{\infty} \right)$. 
  Global regularity then follows from the BKM type criterion in
  {\cite{Caflisch1997}}.
\end{remark}

In light of Lemma \ref{lem:blowup-omega-p}, all we need to do is to show that
when $2 \alpha + \beta > 2$, there is indeed $p > \frac{1}{\alpha}$ such that
$\left\| \omega \right\|_{L^p}$ remains uniformly bounded over $\left( 0, T
\right)$.

Recall the equation for $\omega$:
\begin{equation}
  \omega_t + u \cdot \nabla \omega = b \cdot \nabla j - \Lambda^{\alpha}
  \omega .
\end{equation}
Multiply both sides by $p \left| \omega \right|^{p - 2} \omega$ and integrate
we reach
\begin{equation}
  \frac{\mathd}{\mathd t} \int_{\mathbbm{R}^2} \left| \omega \right|^p \mathd x
  \leqslant p \int_{\mathbbm{R}^2} \left| b \right|  \left| \nabla j \right| 
  \left| \omega \right|^{p - 1} \mathd x - p \int_{\mathbbm{R}^2} \left(
  \Lambda^{\alpha} \omega \right)  \left| \omega \right|^{p - 2} \omega \mathd
  x.
\end{equation}
after taking advantage of $\nabla \cdot u = 0$.

For the dissipation term, it is well-known that
\begin{equation}
  \int_{\mathbbm{R}^2} \left( \Lambda^{\alpha} \omega \right)  \left| \omega
  \right|^{p - 2} \omega \mathd x \geqslant 0.
\end{equation}
This is originally proved in {\cite{Resnick1995}}, and has later been refined
in {\cite{Cordoba2004}}, {\cite{Ju2005}}.

Taking into account the above ``positivity'' property and using H\"older's
inequality, we obtain
\begin{equation}
  \frac{\mathd}{\mathd t} \left\| \omega \right\|_{L^p} \leqslant \left\| b
  \cdot \nabla j \right\|_{L^p} \leqslant \left\| b \right\|_{L^{\infty}} 
  \left\| \nabla j \right\|_{L^p} .
\end{equation}
Now as $\beta > 1$, we have $H^1$ estimate as in \ref{subsec:H1-est}. In
particular we have
\begin{equation}
  j \in L^2  \left( 0, T ; H^{\beta} \right) .
\end{equation}
Sobolev embedding then gives
\begin{equation}
  j \in L^2  \left( 0, T ; H^{\beta} \right) \Longrightarrow \nabla
  j \in L^2 \left( 0, T ; L^p \right)\text{ and }
  b\in L^2(0,T;L^\infty). 
\end{equation}
with $p>\frac{1}{\alpha}$ satisfying
\begin{equation}
  p \leqslant \frac{2}{2 - \beta} \text{ when }\beta<2, \text{ and }
  p < \infty \text{ when }\beta\geqslant 2.
\end{equation}
As $\alpha+\beta>2$, such $p$ exists. Now we have
\begin{equation}
  \left\| \omega \right\|_{L^p} \leqslant \left\| \omega_0 \right\|_{L^p} +
  \int_0^t \left\| b \right\|_{L^{\infty}}  \left\| \nabla j \right\|_{L^p}
  \mathd \tau \leqslant \left\| \omega_0 \right\|_{L^p} + \left\| b
  \right\|_{L^2 \left( 0, T ; L^{\infty} \right)}  \left\| \nabla j
  \right\|_{L^2 \left( 0, T ; L^p \right)} \leqslant C \left( \omega_0, T
  \right) .
\end{equation}
Therefore $\left\| \omega \right\|_{L^p} \in L^{\infty} \left( 0, T \right)$
and global regularity follows from Lemma \ref{lem:blowup-omega-p}.

\section{Proof of Theorem \ref{thm:global} Case III: $\alpha \geqslant 2, \beta =
0$.}\label{Sec:alpha=2}

In this section we prove global regularity in the case $\alpha \geqslant 2,
\beta = 0$. As we identify $\beta = 0$ with $\kappa = 0$, the GMHD equations
now reads
\begin{eqnarray}
  u_t + u \cdot \nabla u & = & - \nabla p + b \cdot \nabla b - 
  \Lambda^{2 \alpha} u, 
  \label{eq:GMHD-u-alpha=2}\\
  b_t + u \cdot \nabla b & = & b \cdot \nabla u,  \label{eq:GMHD-b-alpha=2}\\
  \nabla \cdot u = \nabla \cdot b & = & 0.  \label{eq:GMHD-div-alpha=2}
\end{eqnarray}
In what follows we will only present the proof for the case $\alpha = 2, \beta
= 0$. The case $\alpha > 2$ can be dealt with using the idea in Remark
\ref{rem:reduction}. In fact it can also be proved following standard energy
estimates similar to that in Section \ref{sec:Case I}, as when $\alpha > 2$ we
immediately have $\omega \in L^2 \left( 0, T ; L^{\infty} \right)$. This leads
to a priori $H^1$ bounds which are sufficient to prove a priori $H^2$ bounds.

We will show that when $\alpha\geqslant 2$, the $H^2$ norms of $\omega$ and $j$ must stay finite for any $T>0$. Once this is proved, Sobolev embedding immediately gives the finiteness of $\left\| \omega\right\|_{L^\infty}$ and $\left\| j\right\|_{L^\infty}$ and regularity follows. The $H^2$ bound is proved by contradiction: Assume $\limsup_{t \nearrow T} \left\| \omega
\right\|_{H^2} + \left\| j \right\|_{H^2} = \infty$ for some finite time $T >
0$. The idea is to start from a time $T_0$ close enough to $T$ and show that
under such assumption $\left\| \omega \right\|_{H^2} + \left\| j
\right\|_{H^2}$ remains uniformly bounded for $T_0 < t < T$, thus reaching a
contradiction.

First observe that in this case, energy conservation gives
\begin{equation}
  u, b \in L^{\infty}  \left( 0, T ; L^2 \right), \hspace{2em} \triangle u \in
  L^2 \left( 0, T ; L^2 \right) \Longrightarrow \nabla u, \omega \in L^2 \left( 0, T ;
  \tmop{BMO} \right) \longhookrightarrow L^1 \left( 0, T ; \tmop{BMO} \right)
  .
\end{equation}
\subsection{$H^1$ Estimates}\label{subsec:alpha=2-H1-est}

Similar to Section \ref{subsec:H1-est}, we have
\begin{equation}
  \frac{1}{2}  \frac{\mathd}{\mathd t} \left( \left\| \omega \right\|_{L^2}^2
  + \left\| j \right\|_{L^2}^2 \right) + \left\| \triangle \omega
  \right\|_{L^2}^2 \leqslant \left| \int_{\mathbb{R}^2} T \left( \nabla u, \nabla b \right) j
  \mathd x \right|
\end{equation}
Recalling (\ref{eq:T-H1})
\begin{equation}
  T \left( \nabla u, \nabla b \right) = {\color{black} 2 \partial_1 b_1 
  \left( \partial_1 u_2 + \partial_2 u_1 \right) + 2 \partial_2 u_2  \left(
  \partial_1 b_2 + \partial_2 b_1 \right)}
\end{equation}
and using
\begin{equation}
  \left\| \nabla b \right\|_{L^2} \leqslant C \left\| j \right\|_{L^2},
\end{equation}
we have
\begin{equation}
  \left| \int_{\mathbb{R}^2} jT \left( \nabla u, \nabla b \right) \mathd x \right| \leqslant
  C \left\| \nabla u \right\|_{L^{\infty}}  \left\| j \right\|_{L^2}^2 .
\end{equation}
This gives
\begin{equation}
  \frac{\mathd}{\mathd t} \left( \left\| \omega \right\|_{L^2}^2 + \left\| j
  \right\|_{L^2}^2 \right) + 2 \left\| \triangle \omega \right\|_{L^2}^2
  \leqslant C \left\| \nabla u \right\|_{L^{\infty}} \left( \left\| \omega
  \right\|_{L^2}^2 + \left\| j \right\|_{L^2}^2 \right) .
  \label{eq:H1-alpha=2}
\end{equation}
Here we make use of the following Gronwall-type inequality, which is a variant
of the standard Gronwall's inequality as presented in {\cite{Evans1998}},
Appendix B.j.

\begin{lemma}
  \label{lem:Gronwall}Let $\eta \left( \cdot \right)$ be a nonnegative,
  absolutely continuous function on $\left[ 0, T \right]$, which satisfies for
  a.e. $t$ the inequality
  \begin{equation}
    \eta' \left( t \right) + \psi \left( t \right) \leqslant \phi \left( t
    \right) \eta \left( t \right),
  \end{equation}
  where $\phi \left( t \right)$ and $\psi \left( t \right)$ are nonnegative,
  summable functions on $\left[ 0, T \right]$. Then
  \begin{equation}
    \eta \left( t \right) + \int_0^t \psi \left( \tau \right) \mathd \tau
    \leqslant \eta \left( 0 \right) \exp \left[ \int_0^t \phi \left( \tau
    \right) \mathd \tau \right] .
  \end{equation}
  \begin{proof}
    The proof follows the same idea as that presented in {\cite{Evans1998}}
    and is omitted. 
  \end{proof}
\end{lemma}

Taking $\eta \assign \left\| \omega \right\|_{L^2}^2 + \left\| j
\right\|_{L^2}^2$ and $\psi \assign 2 \left\| \triangle \omega
\right\|_{L^2}^2$ in Lemma \ref{lem:Gronwall}, then integrating from
$T_0$ to $t$, we obtain
\begin{equation}
  \int_{T_0}^t \left\| \triangle \omega \right\|_{L^2}^2 \mathd \tau \leqslant
  \left\| \omega \right\|_{L^2}^2 + \left\| j \right\|_{L^2}^2 + \int_{T_0}^t
  \left\| \triangle \omega \right\|_{L^2}^2 \mathd \tau \leqslant \left( \left\|
  \omega_0 \right\|_{L^2}^2 + \left\| j_0 \right\|_{L^2}^2 \right) \exp \left[
  C \int_{T_0}^t \left\| \nabla u \right\|_{L^{\infty}} \left( \tau \right)
  \mathd \tau \right] .
\end{equation}
Here $T_0 \in \left( 0, T \right)$ will be fixed later and we denote 
$\omega_0:=\omega(\cdot,T_0),\ j_0:=j(\cdot,T_0)$.

Now applying the logarithmic inequality (see e.g. {\cite{Kozono2000a}})
\begin{equation}
  \left\| \nabla u \right\|_{L^{\infty}} \leqslant C \left( 1 + \left\| u
  \right\|_{L^2} + \left\| \omega \right\|_{\tmop{BMO}} \left( 1 + \log \left(
  1 + \left\| \omega \right\|_{H^2}^2 + \left\| j \right\|_{H^2}^2 \right)
  \right) \right) \label{eq:log-ineq}
\end{equation}
and setting
\begin{equation}
  M \left( t \right) \assign \max_{\tau \in \left( T_0, t \right)} \left(
  \left\| \omega \right\|_{H^2}^2 + \left\| j \right\|_{H^2}^2 \right) \left(
  \tau \right) \label{eq:def-M}
\end{equation}
we reach
\begin{equation}
  \int_{T_0}^t \left\| \triangle \omega \right\|_{L^2}^2 \mathd \tau \leqslant
  \left( \left\| \omega_0 \right\|_{L^2}^2 + \left\| j_0 \right\|_{L^2}^2
  \right) \exp \left[ C \left( 1 + \left\| u \right\|_{L^2} \right) \right]
  \exp \left[ C \left( \int_{T_0}^t \left\| \omega \right\|_{\tmop{BMO}}
  \mathd \tau \right) \left( 1 + \log \left( 1 + M \left( t \right) \right) 
  \right) \right] .
\end{equation}
Note that thanks to the energy estimate $\left\| u \right\|_{L^2} \leqslant
\left\| u(0) \right\|_{L^2}$ so $\exp \left( C \left( 1 +
\left\| u \right\|_{L^2} \right) \right)$ is bounded by an constant independent of $T_0$.

As $\left\| \omega \right\|_{\tmop{BMO}} \in L^1 \left( T_0, T \right)$,
we can take $T_0$ close enough to $T$ so that
\begin{equation}
  C \int_{T_0}^t \left\| \omega \right\|_{\tmop{BMO}} \mathd \tau \leqslant 2
  \delta
\end{equation}
for some small positive number $\delta$ to be fixed later. With such choice of
$T_0$ we have
\begin{equation}
  \int_{T_0}^t \left\| \triangle \omega \right\|_{L^2}^2 \mathd \tau \leqslant
  C \left( T_0 \right)  \left( 1 + M \left( t \right) \right)^{2 \delta} .
\end{equation}
Now H\"older's inequality gives
\begin{equation}
  \int_{T_0}^t \left\| \triangle \omega \right\|_{L^2} \mathd \tau \leqslant C
  \left( T_0 \right)  \left( 1 + M \left( t \right) \right)^{\delta} .
  \label{eq:triangle-omega-est}
\end{equation}
Before proceeding, we fix $T_0$ by the following requirements:
\begin{equation}
  C \int_{T_0}^t \left\| \omega \right\|_{\tmop{BMO}} \mathd \tau \leqslant 2
  \delta,\qquad \log(1+M(T_0))>1.
  \label{eq:Cond-T0}
\end{equation}
At the end of Section \ref{subsec:alpha=2-H3-est} we will show that $\delta$ can be
taken as $1 / 24$.

\subsection{$H^3$ estimate ($H^2$ estimate for $\omega,
j$)}\label{subsec:alpha=2-H3-est}

In this subsection we prove the uniform boundedness of $M \left( t \right)$
for all $T_0 < t < T$, thus reaching contradiction.

Let $\partial^2$ denote any double partial derivative (such as
$\partial_{12}, \partial_{11}$ etc.). Taking $\partial^2$ of
(\ref{eq:omega-H1}) and (\ref{eq:j-H1}) and multiplying the resulting 
equations by $\partial^2 \omega$ and $\partial^2 j$ respectively, we reach, 
after using $\nabla \cdot u = \nabla \cdot b = 0$,
\begin{equation}
  \frac{1}{2}  \frac{\mathd}{\mathd t} \int_{\mathbbm{R}^2} \left[ \left( \partial^2
  \omega \right)^2 + \left( \partial^2 j \right)^2 \right] \mathd x \leqslant A + B +
  C + D + E - \int_{\mathbbm{R}^2} \left( \triangle \partial^2 \omega
  \right)^2 \mathd x ,
\end{equation}
with
\begin{eqnarray}
  A & = & \left| \int_{\mathbbm{R}^2} \left[ \partial^2 \left( u \cdot \nabla
  \omega \right) - u \cdot \nabla \partial^2 \omega \right]  \left( \partial^2
  \omega \right) \mathd x \right| \nonumber \\
  &&\leqslant \int_{\mathbbm{R}^2} \left|
  \nabla^2 u \right|  \left| \nabla \omega \right|  \left| \nabla^2 \omega
  \right| \mathd x + \int_{\mathbbm{R}^2} \left| \nabla u \right|  \left|
  \nabla^2 \omega \right|^2 \mathd x ; \\
  B & = & \left| \int_{\mathbbm{R}^2} \left[ \partial^2 \left( b \cdot \nabla
  j \right) - b \cdot \nabla \partial^2 j \right]  \left( \partial^2 \omega
  \right) \mathd x \right| \nonumber\\
  && \leqslant \int_{\mathbbm{R}^2} \left| \nabla^2 b
  \right|  \left| \nabla j \right|  \left| \nabla^2 \omega \right| \mathd x +
  \int_{\mathbbm{R}^2} \left| \nabla b \right|  \left| \nabla^2 j \right| 
  \left| \nabla^2 \omega \right| \mathd x ; \\
  C & = & \left| \int_{\mathbbm{R}^2} \left[ \partial^2 \left( u \cdot \nabla
  j \right) - u \cdot \nabla \left( \partial^2 j \right) \right]  \left(
  \partial^2 j \right) \mathd x \right| \nonumber \\
  && \leqslant \int_{\mathbbm{R}^2} \left|
  \nabla^2 u \right|  \left| \nabla j \right|  \left| \nabla^2 j \right|
  \mathd x + \int_{\mathbbm{R}^2} \left| \nabla u \right|  \left| \nabla^2 j
  \right|^2 \mathd x ; \\
  D & = & \left| \int_{\mathbbm{R}^2} \left[ \partial^2 \left( b \cdot \nabla
  \omega \right) - b \cdot \nabla \left( \partial^2 \omega \right) \right] 
  \left( \partial^2 j \right) \mathd x \right| \nonumber \\
  &&\leqslant \int_{\mathbbm{R}^2}
  \left| \nabla^2 b \right|  \left| \nabla \omega \right|  \left| \nabla^2 j
  \right| \mathd x + \int_{\mathbbm{R}^2} \left| \nabla b \right|  \left|
  \nabla^2 \omega \right|  \left| \nabla^2 j \right. | \mathd x ; \\
  E & = & \left| \int_{\mathbbm{R}^2} \partial^2 T \left( \nabla u, \nabla b
  \right)  \left( \partial^2 j \right) \mathd x \right| \nonumber\\
  & \leqslant & \int_{\mathbbm{R}^2} \left| \nabla^3 u \right|  \left| \nabla
  b \right|  \left| \nabla^2 j \right| \mathd x + \int_{\mathbbm{R}^2} \left|
  \nabla^2 u \right|  \left| \nabla^2 b \right|  \left| \nabla^2 j \right|
  \mathd x + \int_{\mathbbm{R}^2} \left| \nabla u \right|  \left| \nabla^3 b
  \right|  \left| \nabla^2 j \right| \mathd x. 
\end{eqnarray}
Adding up all such partial derivatives, we obtain
\begin{equation}
  \frac{\mathd}{\mathd t} \left( \left\| \nabla^2 \omega \right\|_{L^2}^2 +
  \left\| \nabla^2 j \right\|_{L^2}^2 \right) \leqslant C \left( I_1 +I_{2}
+I_{3} +I_{4}+I_{5}+ I_6 \right) - 2 \left\| \nabla^4 \omega \right\|_{L^2}^2,
\end{equation}
with
\begin{eqnarray}
  I_1 & = & \int_{\mathbbm{R}^2} \left| \nabla^2 u \right|  \left| \nabla
  \omega \right|  \left| \nabla^2 \omega \right| \mathd x ; \\
  I_2 & = & \int_{\mathbbm{R}^2} \left| \nabla u \right|  \left| \nabla^2
  \omega \right|^2 \mathd x + \int_{\mathbbm{R}^2} \left| \nabla u \right| 
  \left| \nabla^2 j \right|^2 \mathd x + \int_{\mathbbm{R}^2} \left| \nabla u
  \right|  \left| \nabla^3 b \right|  \left| \nabla^2 j \right| \mathd x ; \\
  I_3 & = & \int_{\mathbbm{R}^2} \left| \nabla^2 b \right|  \left| \nabla j
  \right|  \left| \nabla^2 \omega \right| \mathd x ; \\
  I_4 & = & \int_{\mathbbm{R}^2} \left| \nabla b \right|  \left| \nabla^2 j
  \right|  \left| \nabla^2 \omega \right| \mathd x + \int_{\mathbbm{R}^2}
  \left| \nabla^3 u \right|  \left| \nabla b \right|  \left| \nabla^2 j
  \right| \mathd x ; \\
  I_5 & = & \int_{\mathbbm{R}^2} \left| \nabla^2 u \right|  \left| \nabla j
  \right|  \left| \nabla^2 j \right| \mathd x ; \\
  I_6 & = & \int_{\mathbbm{R}^2} \left| \nabla^2 b \right|  \left| \nabla
  \omega \right|  \left| \nabla^2 j \right| \mathd x + \int_{\mathbbm{R}^2}
  \left| \nabla^2 u \right|  \left| \nabla^2 b \right|  \left| \nabla^2 j
  \right| \mathd x .
\end{eqnarray}
We remark that the integrals in each $I_k$ can be estimated similarly,
therefore in the following we only show how to estimate the first integral in
each $I_k$.
\begin{itemize}
  \item $I_1$. For $I_1$ we write
  \begin{eqnarray}
    I_1 & \leqslant & \left\| \nabla^2 u \right\|_{L^4}  \left\| \nabla \omega
    \right\|_{L^4}  \left\| \nabla^2 \omega \right\|_{L^2} \nonumber\\
    & \leqslant & C \left\| \nabla \omega \right\|_{L^4}^2  \left\| \nabla^2
    \omega \right\|_{L^2} \nonumber\\
    & \leqslant & C \left\| u \right\|_{L^2}  \left\| \nabla^4 \omega
    \right\|_{L^2}  \left\| \nabla^2 \omega \right\|_{L^2} ,
  \end{eqnarray}
  where we have used the following Gagliardo-Nirenberg inequality
  \begin{equation}
    \left\| \nabla \omega \right\|_{L^4} \leqslant C \left\| u
    \right\|_{L^2}^{1 / 2} \left\| \nabla^4 \omega \right\|_{L^2}^{1 / 2} .
  \end{equation}
  Now by Young's inequality we have, after using $\left\| u \right\|_{L^2}
  \leqslant \left\| u_0 \right\|_{L^2}$,
  \begin{equation}
    I_1 \leqslant C \left( \varepsilon \right) \left\| u \right\|_{L^2}^2 
    \left\| \nabla^2 \omega \right\|_{L^2}^2 + \varepsilon \left\| \nabla^4
    \omega \right\|_{L^2}^2 \leqslant C \left( \varepsilon \right)  \left\|
    \nabla^2 \omega \right\|_{L^2}^2 + \varepsilon \left\| \nabla^4 \omega
    \right\|_{L^2} ,
  \end{equation}
  with $\varepsilon$ as small as necessary.
  
  \item $I_2$. We have
  \begin{eqnarray}
    \int_{\mathbbm{R}^2} \left| \nabla u \right|  \left| \nabla^2 \omega
    \right|^2 \mathd x & \leqslant & \left\| \nabla u \right\|_{L^{\infty}} 
    \left\| \nabla^2 \omega \right\|_{L^2}^2 \nonumber\\
    & \leqslant & C \left( 1 + \left\| u \right\|_{L^2} + \left\| \omega
    \right\|_{\tmop{BMO}}  \left( 1 + \log \left( 1 + \left\| \omega
    \right\|_{H^2}^2 + \left\| j \right\|_{H^2}^2 \right) \right) \right) 
    \left\| \nabla^2 \omega \right\|_{L^2}^2 \nonumber\\
    & \leqslant & C \left( 1 + \left\| \omega \right\|_{\tmop{BMO}}  \left( 1
    + \log \left( 1 + \left\| \omega \right\|_{H^2}^2 + \left\| j
    \right\|_{H^2}^2 \right) \right) \right)  \left\| \nabla^2 \omega
    \right\|_{L^2}^2 ,
  \end{eqnarray}
  where we have used the logarithmic inequality (\ref{eq:log-ineq}).
  
  \item $I_3$. We have
  \begin{eqnarray}
    \int_{\mathbbm{R}^2} \left| \nabla^2 b \right|  \left| \nabla j \right| 
    \left| \nabla^2 \omega \right| \mathd x & \leqslant & \left\| \nabla^2 b
    \right\|_{L^4}  \left\| \nabla j \right\|_{L^4}  \left\| \nabla^2 \omega
    \right\|_{L^2} \nonumber\\
    & \leqslant & C \left\| \nabla j \right\|_{L^4}^2  \left\| \nabla^2
    \omega \right\|_{L^2} \nonumber\\
    & \leqslant & C \left\| b \right\|_{L^2}^{1 / 3}  \left\| \nabla^2 j
    \right\|_{L^2}^{5 / 3}  \left\| \nabla^2 \omega \right\|_{L^2} ,
  \end{eqnarray}
  where we have used the following Gagliardo-Nirenberg inequality  \begin{equation}
    \left\| \nabla j \right\|_{L^4} \leqslant C \left\| b \right\|_{L^2}^{1 /
    6} \left\| \nabla^2 j \right\|_{L^2}^{5 / 6} .
  \end{equation}
  As a consequence (recall the definition of $M \left( t \right)$ in
  (\ref{eq:def-M}))
  \begin{equation}
    I_3 \leqslant C \left\| \nabla^2 \omega \right\|_{L^2} M \left( t
    \right)^{5 / 6} .
  \end{equation}
  Here we have used the energy conservation $\left\| b \right\|_{L^2}
  \leqslant \left\| b_0 \right\|_{L^2} + \left\| u_0 \right\|_{L^2}$.
  
  \item $I_4$. We have
  \begin{eqnarray}
    \int_{\mathbbm{R}^2} \left| \nabla b \right|  \left| \nabla^2 j \right| 
    \left| \nabla^2 \omega \right| \mathd x & \leqslant & \left\| \nabla b
    \right\|_{L^{\infty}}  \left\| \nabla^2 j \right\|_{L^2}  \left\| \nabla^2
    \omega \right\|_{L^2} \nonumber\\
    & \leqslant & C \left\| b \right\|_{L^2}^{1 / 3}  \left\| \nabla^2 j
    \right\|_{L^2}^{5 / 3}  \left\| \nabla^2 \omega \right\|_{L^2} ,
  \end{eqnarray}
  where we have used the following Gagliardo-Nirenberg inequality  \begin{equation}
    \left\| \nabla b \right\|_{L^{\infty}} \leqslant C \left\| b
    \right\|_{L^2}^{1 / 3} \left\| \nabla^2 j \right\|_{L^2}^{2 / 3} .
  \end{equation}
  Therefore
  \begin{equation}
    I_4 \leqslant C \left\| \nabla^2 \omega \right\|_{L^2} M \left( t
    \right)^{5 / 6} .
  \end{equation}
  \item $I_5$. We have
  \begin{eqnarray}
    I_5 & = & \int_{\mathbbm{R}^2} \left| \nabla^2 u \right|  \left| \nabla j
    \right|  \left| \nabla^2 j \right| \mathd x \nonumber\\
    & \leqslant & \left\| \nabla^2 u \right\|_{L^4}  \left\| \nabla j
    \right\|_{L^4}  \left\| \nabla^2 j \right\|_{L^2} \nonumber\\
    & \leqslant & C \left\| u \right\|_{L^2}^{1 / 6}  \left\| \nabla^2 \omega
    \right\|_{L^2}^{5 / 6}  \left\| b \right\|_{L^2}^{1 / 6}  \left\| \nabla^2
    j \right\|_{L^2}^{11 / 6} ,
  \end{eqnarray}
  where we have used the following Gagliardo-Nirenberg inequalities
  \begin{equation}
    \left\| \nabla^2 u \right\|_{L^4} \leqslant C \left\| u \right\|_{L^2}^{1
    / 6}  \left\| \nabla^2 \omega \right\|_{L^2}^{5 / 6} ; \hspace{2em}
    \left\| \nabla j \right\|_{L^4} \leqslant C \left\| b \right\|_{L^2}^{1 /
    6} \left\| \nabla^2 j \right\|_{L^2}^{5 / 6} .
  \end{equation}
  Hence
  \begin{equation}
    I_5 \leqslant C \left\| \nabla^2 \omega \right\|_{L^2}^{5 / 6} M \left( t
    \right)^{11 / 12} \leqslant C \left( 1 + \left\| \nabla^2 \omega
    \right\|_{L^2} \right) M \left( t \right)^{11 / 12} .
  \end{equation}
  \item $I_6$. We have
  \begin{eqnarray}
    \int_{\mathbbm{R}^2} \left| \nabla^2 b \right|  \left| \nabla \omega
    \right|  \left| \nabla^2 j \right| \mathd x & \leqslant & \left\| \nabla^2
    b \right\|_{L^4}  \left\| \nabla \omega \right\|_{L^4}  \left\| \nabla^2 j
    \right\|_{L^2} \nonumber\\
    & \leqslant & C \left\| \nabla j \right\|_{L^4}  \left\| \nabla \omega
    \right\|_{L^4}  \left\| \nabla^2 j \right\|_{L^2} \nonumber\\
    & \leqslant & \left\| b \right\|_{L^2}^{1 / 6}  \left\| u
    \right\|_{L^2}^{1 / 6}  \left\| \nabla^2 \omega \right\|_{L^2}^{5 / 6} 
    \left\| \nabla^2 j \right\|_{L^2}^{11 / 6} ,
  \end{eqnarray}
  where we have used the following Gagliardo-Nirenberg inequalities
  \begin{equation}
    \left\| \nabla \omega \right\|_{L^4} \leqslant C \left\| u
    \right\|_{L^2}^{1 / 6}  \left\| \nabla^2 \omega \right\|_{L^2}^{5 / 6} ;
    \hspace{2em} \left\| \nabla j \right\|_{L^4} \leqslant C \left\| b
    \right\|_{L^2}^{1 / 6} \left\| \nabla^2 j \right\|_{L^2}^{5 / 6}.
  \end{equation}
  Hence
  \begin{equation}
    I_6 \leqslant C \left\| \nabla^2 \omega \right\|_{L^2}^{5 / 6} M \left( t
    \right)^{11 / 12} \leqslant C \left( 1 + \left\| \nabla^2 \omega
    \right\|_{L^2} \right) M \left( t \right)^{11 / 12} .
  \end{equation}
\end{itemize}
Summarizing, we have
\begin{eqnarray}
  \frac{\mathd}{\mathd t} \left( \left\| \nabla^2 \omega \right\|_{L^2}^2 +
  \left\| \nabla^2 j \right\|_{L^2}^2 \right) & \leqslant & C \left( T_0 \right) 
  \left[ M \left( t \right) + \left( 1 + \left\| \nabla^2 \omega
  \right\|_{L^2} \right) M \left( t \right)^{11 / 12}\right. \nonumber \\ 
  && + \left.\left\| \omega
  \right\|_{\tmop{BMO}} M \left( t \right) \log \left( 1 + M \left( t \right)
  \right) \right].
\end{eqnarray}
Using our assumption on $T_0$ (\ref{eq:Cond-T0}) and the monotonicity of $M(t)$, we have $\log \left(
1 + M \left( t \right) \right) > 1$ and therefore
\begin{eqnarray}
  \frac{\mathd}{\mathd t} \left( \left\| \nabla^2 \omega \right\|_{L^2}^2 +
  \left\| \nabla^2 j \right\|_{L^2}^2 \right) &\leqslant& C\left( T_0 \right)
  \left[ \left( 1 + \left\| \nabla^2 \omega \right\|_{L^2} \right) M \left( t
  \right)^{11 / 12} \right. \nonumber \\
  &&+ \left. \left(1+\left\| \omega \right\|_{\tmop{BMO}}\right)
   M \left( t \right)
  \log \left( 1 + M \left( t \right) \right) \right] .
\end{eqnarray}
Integrating, we have
\begin{eqnarray}
  M \left( t \right) & \leqslant& C \left( T_0 \right) \left[ M_0 + \left(
  \int_{T_0}^t \left( 1 + \left\| \nabla^2 \omega \right\|_{L^2} \right)
  \mathd \tau \right) M \left( t \right)^{11 / 12}\right. \nonumber \\
  && + \left.\int_{T_0}^t \left[
  \left(1+
  \left\| \omega \right\|_{\tmop{BMO}}\right)
   M \left( \tau \right) \log \left( 1 + M
  \left( \tau \right) \right) \right] \mathd \tau \right] ,
\end{eqnarray}
with $M_0 \assign \left\| \omega \right\|_{H^2}^2 \left( T_0 \right) + \left\|
j \right\|_{H^2}^2 \left( T_0 \right)$.

Now taking $\delta = 1 / 24$, we have
\begin{equation}
  \int_{T_0}^t \left( 1 + \left\| \nabla^2 \omega \right\|_{L^2} \right)
  \mathd \tau \leqslant C \left( T_0 \right) \left( 1+ M \left( t \right)\right)^{1 / 24} ,
\end{equation}
which leads to
\begin{eqnarray}
  M \left( t \right) &\leqslant& C \left( T_0 \right) \left[ M_0 + M \left( t
  \right)^{11 / 12} \left(1+M(t)\right)^{1/24}\right. \nonumber \\
  && + \left.\int_{T_0}^t \left[ \left(1+\left\| \omega \right\|_{\tmop{BMO}}\right)
  M \left( \tau \right) \log \left( 1 + M \left( \tau \right) \right) \right]
  \mathd \tau \right] .
\end{eqnarray}
This in turn gives
\begin{eqnarray}
  1 + M \left( t \right) &\leqslant & C \left( T_0 \right)  \left[ \left( 1 + M_0
  \right) + \left( 1 + M \left( t \right) \right)^{23 / 24} \right. \nonumber \\
&& + \left.\int_{T_0}^t
  \left[ \left(1+\left\| \omega \right\|_{\tmop{BMO}}\right)
    \left( 1 + M \left( \tau
  \right) \right) \log \left( 1 + M \left( \tau \right) \right) \right] \mathd
  \tau \right] .
\end{eqnarray}
Now we set $N \left( t \right) \assign \left( 1 + M \left( t \right)
\right)^{1 / 24}$, $N_0 \assign \left( 1 + M_0 \right)^{1 / 24}$ and divide
both sides by $\left( 1 + M \left( t \right) \right)^{23 / 24}$, using the
monotonicity of $M \left( t \right)$ we reach
\begin{equation}
  N \left( t \right) \leqslant C \left( T_0 \right)  \left[ \left( 1 + N_0
  \right) + \int_{T_0}^t \left(1+\left\| \omega \right\|_{\tmop{BMO}}\right)
   N \left( \tau
  \right) \log \left( N \left( \tau \right) \right) \mathd \tau \right] .
\end{equation}
Application of the standard Gronwall's inequality now gives the following bound of $N$ 
\begin{equation}
  N \left( t \right) \leqslant \left[ C \left( T_0 \right)  \left( 1 + N_0
  \right) \right]^{\exp \left[ C \left( T_0 \right)  \int_{T_0}^t \left(1+
  \left\|
  \omega \right\|_{\tmop{BMO}}\right)
   \mathd \tau \right]} ,
\end{equation}
which gives
\begin{equation}
  M \left( t \right) \leqslant \left[ C \left( T_0 \right)  \left( 1 + N_0
  \right) \right]^{24 \exp \left[ C \left( T_0 \right)  \int_{T_0}^t
   \left(1+\left\|
  \omega \right\|_{\tmop{BMO}}\right)
   \mathd \tau \right]} .
  \label{eq:M-est}
\end{equation}
Since $\int_{T_0}^t \left\| \omega \right\|_{\tmop{BMO}} \left( \tau \right)
\mathd \tau$ remains bounded as $t \nearrow T$, (\ref{eq:M-est}) contradicts
our assumption that $M \left( t \right) \nearrow \infty$ as $t \nearrow T$ and
ends the proof.

\subsection{$H^k$ estimate}\label{subsec:alpha=2-Hk-est}

As we have already proved that the $H^2$ norms of $\omega$ and $j$ have to remain bounded as $t\nearrow T$, thanks to the embedding $H^2 \longhookrightarrow L^{\infty}$ in
$\mathbbm{R}^2$, we have
\begin{equation}
  \omega, j \in L^{\infty} \left( 0, T ; L^{\infty} \right)
\end{equation}
as a result of the argument in \ref{subsec:alpha=2-H1-est} and
\ref{subsec:alpha=2-H3-est}. The $H^k$ estimate and global regularity is now a
simple consequence of the BKM-type criterion in {\cite{Caflisch1997}}.

\section{Global regularity when the magnetic lines are
smooth}\label{sec:b_good}

This section proves Theorem \ref{thm:b_good}, which states that the system
\begin{eqnarray}
  u_t + u \cdot \nabla u & = & - \nabla p + b \cdot \nabla b, 
  \label{eq:GMHD-u-bgood}\\
  b_t + u \cdot \nabla b & = & b \cdot \nabla u - \Lambda^{2 \beta} b, 
  \label{eq:GMHD-b-bgood}\\
  \nabla \cdot u = \nabla \cdot b & = & 0,  \label{eq:GMHD-div-bgood}
\end{eqnarray}
with $\beta > 1$ and \ $\left( u_0, b_0 \right) \in H^k$ for some $k > 2$, is
globally regular if $\widehat{b} \assign \frac{b}{\left| b \right|} \in L^{\infty}
\left( 0, T ; W^{2, \infty} \right)$.

\begin{proof}
  As $\beta > 1$, following Lemma \ref{lem:H1-est} we already have $H^1$
  estimate which in particular gives
  \begin{equation}
    j \in L^2 \left( 0, T ; H^{\beta} \right) \longhookrightarrow L^2 \left(
    0, T ; L^{\infty} \right)
  \end{equation}
  since $H^{\beta} \longhookrightarrow L^{\infty}$. Thanks to the BKM-type
  criteria in {\cite{Caflisch1997}}, all we need to prove is that $\omega \in
  L^1 \left( 0, T ; L^{\infty} \right)$.
  
  For a proof of $\omega \in L^1 \left( 0, T ; L^{\infty} \right)$, 
  let us examine the vorticity equation
  \begin{equation}
    \omega_t + u \cdot \nabla \omega = \nabla^{\bot} \cdot \left( b \cdot
    \nabla b \right),
  \end{equation}
  where the ``forcing'' term has been given in its raw form instead of 
  $b \cdot \nabla j$ for the very purpose of this proof. By writing
  \begin{equation}
    b = \widehat{b}  \left| b \right|
  \end{equation}
  and using the divergence free condition $\nabla \cdot b = 0$, we have
  \begin{equation}
    \widehat{b} \cdot \nabla \left| b \right| = - \left( \nabla \cdot \widehat{b}
    \right)  \left| b \right| .
  \end{equation}
  It follows that
  \begin{equation}
    b \cdot \nabla b = \left| b \right|  \left[ \widehat{b} \cdot \nabla \left(
    \widehat{b}  \left| b \right| \right) \right] = \left[ \widehat{b} \cdot \nabla
    \widehat{b} - \left( \nabla \cdot \widehat{b} \right)  \widehat{b} \right]  \left| b
    \right|^2 .
  \end{equation}
  Therefore the vorticity equation can be written as
  \begin{equation}
    \omega_t + u \cdot \nabla \omega = \nabla^{\bot} \cdot \left\{ \left[
    \widehat{b} \cdot \nabla \widehat{b} - \left( \nabla \cdot \widehat{b} \right) 
    \widehat{b} \right]  \left| b \right|^2 \right\} = A \left( x, t \right) 
    \left| b \right|^2 + B \left( x, t \right) \cdot \left( b \cdot
    \nabla^{\bot} b \right),
  \end{equation}
  where
  \begin{equation}
    A \left( x, t \right) = \nabla^{\bot} \cdot \left[ \widehat{b} \cdot \nabla
    \widehat{b} - \left( \nabla \cdot \widehat{b} \right)  \widehat{b} \right],
    \hspace{2em} B \left( x, t \right) = \widehat{b} \cdot \nabla \widehat{b} - \left(
    \nabla \cdot \widehat{b} \right)  \widehat{b} .
  \end{equation}
  As $\widehat{b} \in W^{2, \infty}$ by our assumption, we readily deduce that
  \begin{equation}
    A \left( x, t \right), B \left( x, t \right) \in L^{\infty} \left( 0, T ;
    L^{\infty} \right) .
  \end{equation}
  Now since $\beta > 1$, the earlier estimates in $H^1$ mean 
  \begin{equation}
    j \in L^2 \left( 0, T ; H^{\beta} \right) \Longrightarrow \nabla b \in L^2
    \left( 0, T ; H^{\beta} \right) \Longrightarrow \nabla b \in L^2 \left( 0,
    T ; L^{\infty} \right) .
  \end{equation}
  It follows that
  \begin{equation}
    \left| b \right|^2, b \cdot \nabla^{\bot} b \in L^1 \left( 0, T ;
    L^{\infty} \right) .
  \end{equation}
  Putting things together, we see that
  \begin{equation}
    \omega_t + u \cdot \nabla \omega = F \left( x, t \right) \assign A \left(
    x, t \right)  \left| b \right|^2 + B \left( x, t \right) \cdot \left( b
    \cdot \nabla^{\bot} b \right),
  \end{equation}
  with $F \left( x, t \right) \in L^1 \left( 0, T ; L^{\infty} \right)$. Since
  we are dealing with smooth solutions here, this immediately leads to
  \begin{equation}
    \omega \in L^{\infty} \left( 0, T ; L^{\infty} \right) \longhookrightarrow
    L^1 \left( 0, T ; L^{\infty} \right)
  \end{equation}
  and the proof is completed. 
\end{proof}

\begin{remark}
  For solutions not smooth enough, we can argue as follows. First note that $j
  \in L^2 \left( 0, T ; H^{\beta} \right)$ implies $\nabla b \in L^2 \left( 0,
  T ; L^q \right)$ for any $q$, and furthermore $\left\| \nabla b
  \right\|_{L^2 \left( 0, T ; L^q \right)}$ is uniformly bounded in $q$.
  Consequently $F \in L^1 \left( 0, T ; L^q \right)$ for any $q$ with 
  uniformly bounded norms. Now we multiply the equation by 
  $\left| \omega \right|^{p - 2} \omega$ and integrate. After 
  simplification we get
  \begin{equation}
    \frac{\mathd}{\mathd t} \left\| \omega \right\|_{L^p}^p \leqslant p \left|
    \int_{\mathbbm{R}^2} F \left( x, t \right)  \left| \omega \right|^{p - 2}
    \omega \mathd x \right| \leqslant p \left\| F \right\|_{L^p}  \left\|
    \omega \right\|_{L^p}^{p - 1},
  \end{equation}
  which implies
  \begin{equation}
    \frac{\mathd}{\mathd t} \left\| \omega \right\|_{L^p} \leqslant \left\| F
    \right\|_{L^p} .
  \end{equation}
  This gives a uniform bound on $\left\| \omega \right\|_{L^p}$ and
  consequently a bound on $\left\| \omega \right\|_{L^{\infty}}$.
\end{remark}

\paragraph{\textbf{Acknowledgment}}X. Yu and Z. Zhai are supported by a grant from
NSERC and the Startup grant from Faculty of Science of University of Alberta. The authors would like to thank the anonymous referee for the valuable comments and suggestions.

\end{document}